\documentclass[11pt]{article}

\usepackage{amsmath}
\usepackage{amssymb}
\usepackage{amsthm}
\usepackage{enumitem}
\usepackage{titlesec}

\usepackage{comment}

\usepackage{indentfirst}

\newtheorem{thm}{Theorem}[section]
\newtheorem{lemma}[thm]{Lemma}
\newtheorem{prop}[thm]{Proposition}
\newtheorem{conj}[thm]{Conjecture}
\newtheorem{cor}[thm]{Corollary}
\theoremstyle{definition}
\newtheorem{defin}[thm]{Definition}

\numberwithin{equation}{section}

\author{\uppercase{\scriptsize{Riccardo W. Maffucci}}}
\title{\normalsize{\uppercase{\bf{Restriction of 3D arithmetic Laplace eigenfunctions to a plane}}}}
\date{}
\newcommand{\Addresses}{{
  \bigskip
  \footnotesize

  R.W.~Maffucci, \textsc{Mathematical Institute, University of Oxford, Woodstock Road Oxford OX2 6GG, UK}\par\nopagebreak
  \texttt{riccardo.maffucci@maths.ox.ac.uk}
}}

\begin{document}
\titleformat{\section}
  {\Large\scshape\centering\bf}{\thesection}{1em}{}
\titleformat{\subsection}
  {\large\scshape\bf}{\thesubsection}{1em}{}
\maketitle
\begin{abstract}
We consider a random Gaussian ensemble of Laplace eigenfunctions on the 3D torus, and investigate the 1-dimensional Hausdorff measure (`length') of nodal intersections against a smooth 2-dimensional toral sub-manifold (`surface'). The expected length is universally proportional to the area of the reference surface, times the wavenumber, independent of the geometry.

For surfaces contained in a plane, we give an upper bound for the nodal intersection length variance, depending on the arithmetic properties of the plane. The bound is established via estimates on the number of lattice points in specific regions of the sphere.
\end{abstract}


\section{Introduction}
\subsection{Nodal sets for eigenfunctions of the Helmholtz equation}
Let $\Delta_\mathcal{M}$ be the Laplace-Beltrami operator, or for short Laplacian, on a smooth manifold $\mathcal{M}$ of dimension $d$. With motivation coming from physics and PDEs, one is interested in eigenfunctions $G$ of the Helmholtz equation
\begin{equation*}
(\Delta_\mathcal{M}+E)G=0
\end{equation*}
with eigenvalue (or `energy' in the physics terminology) $E>0$, in the high energy limit $E\to\infty$.

Of particular importance is the {\bf nodal set} (zero-locus) of $G$, 
\begin{equation}
\label{nodset}
\mathcal{A}_G:=\{x\in\mathcal{M} : G(x)=0\}.
\end{equation}
Its study dates back to Hooke's and Chladni's pioneering work (17th-18th century). There is a wide range of scientific applications including telecommunications \cite{rice44}, oceanography \cite{longue, azawsc}, and photography \cite{swerli}.

It is known that $\mathcal{A}_G$ is a smooth sub-manifold of dimension $d-1$ except for a set of lower dimension \cite[Theorem 2.2]{cheng1}. For $d=2$, we call $\mathcal{A}_G$ {\bf nodal line}, and for $d=3$, we call it {\bf nodal surface}.

Our setting is the three-dimensional standard flat torus $\mathcal{M}=\mathbb{T}^3=\mathbb{R}^3/\mathbb{Z}^3$. Here the Laplace eigenvalues {`energy levels'}, are of the form $4\pi^2 m$, $m\in S_3$, where
\begin{equation*}
S_3:=\{0<m: \ m=a_1^2+a_2^2+a_3^2,\  a_i\in\mathbb{Z}\}.
\end{equation*}
The frequencies
\begin{equation}
\label{Lambda}
\Lambda_m=\{\lambda\in\mathbb{Z}^3 : \|\lambda\|^2=m\}
\end{equation}
are the lattice points on $\sqrt{m}\mathcal{S}^2$, the sphere of radius $\sqrt{m}$. The (complex-valued) Laplace eigenfunctions may be written as \cite{brnoda}
\begin{equation}
\label{lapeig}
G(x)=G_m(x)=\sum_{\lambda\in\Lambda}
c_{\lambda}
e^{2\pi i\langle\lambda,x\rangle},
\qquad x\in\mathbb{T}^3,
\end{equation}
with $c_\lambda$ Fourier coefficients. 

The eigenspace dimension is the lattice point number, i.e., the number of ways to express $m$ as a sum of three integer squares
\begin{equation}
\label{N}
N:=|\Lambda|=r_3(m).
\end{equation}
In what follows we will always make the (natural) assumption $m\not\equiv 0,4,7 \pmod 8$, implying
\begin{equation}
\label{totnumlp3}
(\sqrt{m})^{1-\epsilon}
\ll
N
\ll
(\sqrt{m})^{1+\epsilon}
\end{equation}
for all $\epsilon>0$ \cite[\S 1]{bosaru} and in particular $N\to\infty$. This assumption is natural in the sense that if $m\equiv 7 \pmod 8$ then $m\not\in S_3$, while multiplying $m$ by $4$ just rescales the frequency set \cite[\S 1.3]{ruwiye}. Further details on the structure of $\Lambda_m$ may be found in section \ref{seclp}.

\subsection{Nodal intersections}
One insightful approach to the study of the nodal set is given by its restriction to a fixed sub-manifold in the ambient $\mathcal{M}$, the so-called {\bf nodal intersections}. The recent papers \cite{totzel,cantot,elhtot} analyse nodal intersections on `generic' surfaces (i.e. $d=2$) against a curve. Unless the curve is contained in the nodal line, the intersection is a set of points. It is expected that in many situations, the nodal intersections number obeys the bound $\ll\sqrt{E}$, where $E>0$ is the eigenvalue.

The nodal set of $G_m$ \eqref{lapeig} is a nodal surface on $\mathbb{T}^3$. We consider the restriction of $G_m$ to a fixed smooth $2$-dimensional sub-manifold $\Pi\subset\mathbb{T}^3$, and specifically the {\bf nodal intersection length}
\begin{equation*}
h_1(\mathcal{A}_G \cap \Pi)
\end{equation*}
where $h_1$ is $1$-dimensional Hausdorff measure, in the high energy limit $m\to\infty$. Bourgain and Rudnick found that, for $\Pi$ real-analytic, with nowhere zero Gauss-Kronecker curvature, there exists $m_\Pi$ such that for every $m\geq m_\Pi$, the surface $\Pi$ is not contained in the nodal set of any eigenfunction $G_m$ \cite[Theorem 1.2]{brnoda}. Moreover, one has the upper bound
\begin{equation}
\label{BRthm}
h_1(\mathcal{A}_G \cap \Pi)<C_\Pi\cdot\sqrt{m}
\end{equation}
for some constant $C_\Pi$ \cite[Theorem 1.1]{brgafa}, and for every eigenfunction $G_m$ the nodal intersection is non-empty \cite[Theorem 1.3]{brgafa}.

\subsection{The arithmetic waves}
\label{secarw}
The eigenvalue multiplicities allow us to randomise our setting as follows. We will be working with an ensemble of {\em random} Gaussian Laplace toral eigenfunctions (`arithmetic waves' for short \cite{orruwi, rudwi2, krkuwi})
\begin{equation}
\label{arw}
F(x)=F_m(x)=\frac{1}{\sqrt{N}}
\sum_{\lambda\in\Lambda}
a_{\lambda}
e^{2\pi i\langle\lambda,x\rangle},
\qquad x\in\mathbb{T}^3,
\end{equation}
of eigenvalue $4\pi^2m$, where $a_{\lambda}$ are complex standard Gaussian random variables \footnote{Defined on some probability space $(\Omega,\mathcal{F},\mathbb{P})$, where $\mathbb{E}$ denotes
the expectation with respect to $\mathbb{P}$.} (i.e., one has $\mathbb{E}[a_{\lambda}]=0$ and $\mathbb{E}[|a_{\lambda}|^2]=1$), independent save for the relations $a_{-\lambda}=\overline{a_{\lambda}}$ (so that $F(x)$ is real valued). The total area of the nodal surface of $F$ was studied in \cite{benmaf,cammar}. The arithmetic wave \eqref{arw} may be analogously defined on the $d$-dimensional torus $\mathbb{R}^d/\mathbb{Z}^d$. Several recent papers investigate the nodal volume \cite{rudwi2, krkuwi} and nodal intersections of arithmetic waves against a fixed {\em curve} \cite{rudwig, maff2d, roswig, ruwiye, maff3d}.

\subsection{Restriction to a surface of nowhere vanishing Gauss-Kronecker curvature}
\label{prior}
In \cite{maff18} we considered the {\bf nodal intersection length}, i.e. the random variable
\begin{equation}
\label{L}
\mathcal{L}=\mathcal{L}_m:=h_1(\mathcal{A}_{F_m}\cap\Pi)
\end{equation}
where $\Pi$ is a smooth $2$-dimensional sub-manifold of $\mathbb{T}^3$, possibly with boundary, admitting a smooth normal vector locally.
The expected intersection length is $\mathbb{E}[\mathcal{L}]=\sqrt{m}A\pi/\sqrt{3}$, where $A$ is the total area of $\Pi$ \cite[Proposition 1.2]{maff18}. This expectation is independent of the geometry, and is consistent with \eqref{BRthm}.

The main result of \cite{maff18} is the precise asymptotic of the nodal intersection length variance, against surfaces of nowhere vanishing Gauss-Kronecker curvature \cite[Theorem 1.3]{maff18}
\begin{equation}
\label{var}
\text{Var}(\mathcal{L})=\frac{\pi^2}{60}\frac{m}{N}\left[3\mathcal{I}-A^2+O\left(m^{-1/28+o(1)}\right)\right]
\end{equation}
where
\begin{equation*}
\mathcal{I}=\mathcal{I}_\Pi:=\iint_{\Pi^2}\langle\overrightarrow{n}(p),\overrightarrow{n}(p')\rangle^2dpdp'
\end{equation*}
and $\overrightarrow{n}(p)$ is the unit normal vector to $\Pi$ at the point $p$. 

In this paper, we consider the {\bf other extreme} of the nowhere vanishing curvature scenario, namely, the case where $\Pi$ is {\bf contained in a plane}. The above result for the expected intersection length is valid in this case also. The integral $\mathcal{I}$ satisfies the sharp bounds \cite[Proposition 1.4]{maff18}
\begin{equation*}
\frac{A^2}{3}\leq\mathcal{I}\leq A^2,
\end{equation*}
so that the leading coefficient of \eqref{var} is always non-negative and bounded, though it may vanish, for instance when $\Pi$ is a sphere or a hemisphere \footnote{There are also (several) other examples of these so-called `static' surfaces. To establish the variance asymptotic for these seems to be a difficult problem.}: in this case the variance is of lower order than $m/N$. This behaviour is similar to the two-dimensional case \cite{rudwig, roswig}.

The theoretical maximum of the variance asymptotic is achieved in the case of intersection with a surface contained in a plane. Although this case is excluded by the assumptions of \eqref{var}, it is natural to conjecture $\text{Var}(\mathcal{L})\sim A^2m/N\cdot\pi^2/30$ for $\Pi$ confined to a plane.

\subsection{Main results}
\label{secresults}
Let $\Pi$ be a smooth $2$-dimensional sub-manifold of $\mathbb{T}^3$ contained in a plane. We denote $\overrightarrow{n}$ the unit normal vector to this plane. We distinguish between vectors/planes of the following three types, possibly after relabelling the coordinates and assuming w.l.o.g. that $n_1\neq 0$:
\begin{align}
\label{qq}
\tag{i}
{n_2}/{n_1}\in\mathbb{Q} \quad&\text{and}\quad {n_3}/{n_1}\in\mathbb{Q};
\\
\label{qr}
\tag{ii}
{n_2}/{n_1}\in\mathbb{Q} \quad&\text{and}\quad {n_3}/{n_1}\in\mathbb{R}\setminus\mathbb{Q};
\\
\label{rr}
\tag{iii}
{n_2}/{n_1}\in\mathbb{R}\setminus\mathbb{Q} \quad&\text{and}\quad {n_3}/{n_1}\in\mathbb{R}\setminus\mathbb{Q}.
\end{align}
Vectors/planes of type \eqref{qq} will also be called `rational', and the remaining types `irrational'. This terminology is borrowed from \cite{maff3d}.

As in {\cite[\S 2.3]{brgafa}} we will denote $\kappa(R)$ the maximal number of lattice points in the intersection of $R\mathcal{S}^{2}$ and any plane. The upper bound
\begin{equation}
\label{kappa3bound}
\kappa(R)\ll R^\epsilon, \quad \forall\epsilon>0
\end{equation}
is due to Jarnik \cite{jarnik}, \cite[(2.6)]{brgafa}.

\begin{thm}
\label{thmpl}
Let $\Pi$ be a smooth $2$-dimensional sub-manifold of $\mathbb{T}^3$ contained in a plane.
\begin{enumerate}[label=(\arabic*)]
\item
\label{thmpl1}
If the plane is rational, then the nodal intersection length variance satisfies the bound
\begin{equation}
\label{varplr}
\text{Var}(\mathcal{L})\ll_{\Pi}\frac{m}{N}\cdot\kappa(\sqrt{m}).
\end{equation}
\item
\label{thmpl2}
Moreover, for irrational planes we have
\begin{equation}
\label{varpli}
\text{Var}(\mathcal{L})\ll_{\Pi}\frac{m}{N}\cdot N^{a+\epsilon}
\end{equation}
for any positive $\epsilon$ where we may take:
\begin{enumerate}[label=(\Alph*)]
\item
$a=3/7$ for planes of type \eqref{qr};
\item
$a=3/4$ for planes of type \eqref{rr}.
\end{enumerate}
\end{enumerate}
\end{thm}
Theorem \ref{thmpl} will be proven in section \ref{secpl}. Taking into account \eqref{kappa3bound}, the bound \eqref{varplr} is just $\epsilon$'s off from the conjectured order $m/N$. Similarly to \cite{rudwig, ruwiye, maff18}, the above results on expectation and variance have the following consequence.
\begin{thm}
Let $\Pi$ be a smooth $2$-dimensional sub-manifold of $\mathbb{T}^3$ contained in a plane, of total area $A$. Then the nodal intersection length $\mathcal{L}$ satisfies, for all $\epsilon>0$,
\begin{equation*}
\lim_{\substack{m\to\infty \\ m\not\equiv 0,4,7 \pmod 8}}\mathbb{P}\left(\left|\frac{\mathcal{L}}{\sqrt{m}}-\frac{\pi}{\sqrt{3}}A\right|>\epsilon\right)=0.
\end{equation*}
\end{thm}
\begin{proof}
Apply the Chebychev-Markov inequality together with Theorem \ref{thmpl} and \cite[Proposition 1.2]{maff18}.
\end{proof}

Furthermore, one may improve on Theorem \ref{thmpl} conditionally on the following conjecture.
\begin{conj}[Bourgain and Rudnick {\cite[\S 2.2]{brgafa}}]
\label{brgafaconj}
Let $\chi(R,s)$ be the maximal number of lattice points in a cap of radius $s$ of the sphere $R\mathcal{S}^2$. Then for all $\epsilon>0$ and $s<R^{1-\delta}$,
\begin{gather*}
\chi(R,s)
\ll
R^\epsilon\left(1+\frac{s^2}{R}\right)
\end{gather*}
as $R\to\infty$.
\end{conj}
We have the following conditional improvement for planes of type \eqref{rr}.
\begin{thm}
\label{thmc}
Let $\Pi$ be a smooth $2$-dimensional sub-manifold of $\mathbb{T}^3$ contained in a plane. Assuming Conjecture \ref{brgafaconj}, we have for every $\epsilon>0$
\begin{equation}
\label{varplc}
\text{Var}(\mathcal{L})\ll_{\Pi}\frac{m}{N}\cdot N^{1/2+\epsilon}.
\end{equation}
\end{thm}
\noindent
Theorem \ref{thmc} will be proven in section \ref{secpl}.

\subsection{Outline of proofs and plan of the paper}
\label{secout}
The arithmetic random wave $F$ \eqref{arw} is a {\em random field}. For a smooth random field $P:T\subset_{\text{open}}\mathbb{R}^d\to\mathbb{R}^{d'}$, denote $\mathcal{V}$ the Hausdorff measure of its nodal set. For instance when $d=3$ and $d'=1$ then $\mathcal{V}$ is the nodal area. Only the case $d\geq d'$ is interesting, since otherwise the zero set of $P$ is a.s. \footnote{The expression `almost surely', or for short `a.s.', means `with probability $1$'.} empty. Under appropriate assumptions, the moments of $\mathcal{V}$ may be computed via Kac-Rice formulas \cite[Theorems 6.8 and 6.9]{azawsc}. These formulas, however, do not apply to our situation \cite[Example 1.6]{maff18} (except in the very special case of the plane containing $\Pi$ being parallel to one of the coordinate planes). To resolve this issue, in \cite{maff18} we derived Kac-Rice formulas for a random field defined on a {\em surface}, and thus computed $\mathbb{E}[\mathcal{L}]$.

Via an {\bf approximate Kac-Rice formula} \cite[Proposition 1.7]{maff18}, for surfaces of nowhere vanishing Gauss-Kronecker curvature, the problem of computing the nodal intersection length variance \eqref{var} was reduced to estimating the second moment of the {\em covariance function} 
\begin{equation}
\label{rintro}
r(p,p'):=\mathbb{E}[F(p)F(p')]
\end{equation}
and of its various first and second order derivatives. The error term in \eqref{var} comes from bounding the fourth moment of $r$ and of its derivatives.

For $\Pi$ confined to a plane, we wish to prove the upper bounds in Theorem \ref{thmpl}. An {\bf approximate Kac-Rice bound} will then suffice, similarly to \cite{maff2d, ruwiye, maff3d}.
\begin{prop}[Approximate Kac-Rice bound]
\label{approxKRpl}
Let $\Pi$ be a smooth $2$-dimensional sub-manifold of $\mathbb{T}^3$ contained in a plane. Then we have
\begin{equation}
\label{varpl1}
\text{Var}(\mathcal{L})\ll m\iint_{\Pi^2}
\left(
r^2
+\frac{D\Omega D^T}{m}
+\frac{tr(H\Omega H\Omega)}{m^2}
\right)
dpdp'
\end{equation}
where $D(p,p'),H(p,p'),\Omega$ are appropriate vectors and matrices, depending on $r(p,p')$, its derivatives, and $\Pi$ \footnote{See \cite[Definition 3.3]{maff18}.}.
\end{prop}

Proposition \ref{approxKRpl} will be proven in section \ref{secappkr}. The problem of bounding the variance of $\mathcal{L}$ is thus reduced to estimating the second moment of the covariance function $r$ and its various first and second order derivatives. This, in turn, requires estimates for the number of lattice points in specific regions of the sphere $\sqrt{m}\mathcal{S}^2$, covered in section \ref{capseg}.

There are marked differences compared to the case of generic surfaces: first, if $\Pi$ is contained in a plane of unit normal $\overrightarrow{n}=(n_1,n_2,n_3)$, it admits everywhere the parametrisation
\begin{align}
\notag
\gamma:U\subset\mathbb{R}^2&\to\Pi,
\\
\label{gammapl}
(u,v)&\mapsto(P+u\xi+v\eta),
\end{align}
where $P\in\Pi$ and $\{\overrightarrow{n},\xi,\eta\}$ is an orthonormal basis of $\mathbb{R}^3$ \cite[\S 2.5, Example 1]{docarm}. Then the covariance function \eqref{rintro} has the special form
\begin{equation}
\label{rplane}
r((u,v),(u',v'))
=\frac{1}{N}\sum_{\lambda\in\Lambda} e^{2\pi i\langle\lambda,(u'-u)\xi+(v'-v)\eta\rangle},
\end{equation}
depending on the difference $(u',v')-(u,v)$ only: the random field $f(u,v):=F(\gamma(u,v))$ is {\em stationary} \footnote{In particular we may assume w.l.o.g. that $P$ is the origin.}. This behaviour is very different from the case of generic surfaces. In particular it eventually leads to a different method from \cite{maff18} of controlling the second moment, and specifically the off-diagonal terms. Indeed, in our previous paper, the off-diagonal terms are handled via a generalisation of Van der Corput's lemma to higher dimensions \cite[Proposition 5.4]{maff18}, applicable for surfaces $\Pi$ of nowhere vanishing Gauss-Kronecker curvature. On the other hand if $\Pi$ is confined to a plane, the special form \eqref{rplane} of the covariance function allows us to establish the estimates \eqref{min} directly, leading to a different arithmetic problem from the generic surfaces case.

Similarly to \cite{maff2d,maff3d} (nodal intersections against a straight line in two and three dimensions), in the linear case the variance upper bounds depend on the arithmetic properties of the line/plane. In Theorem \ref{thmpl}, the upper bound is stronger in the case of rational planes, and the bound for planes of type \eqref{qr} is stronger than for those of type \eqref{rr}, again similar to \cite{maff2d,maff3d}. This situation occurs because the bounds rely on estimates for lattice points in specific regions of the sphere: when
\begin{equation*}
\frac{n_3}{n_1}, \ \frac{n_3}{n_2}
\end{equation*}
are irrational numbers, the lattice point estimates are derived using simultaneous Diophantine approximation, so that the bound for the variance is stronger when the number of irrationals to approximate is smaller \cite[\S 8]{maff3d}.

\subsection{Acknowledgements}
The author worked on this project mainly during his PhD studies, under the supervision of Igor Wigman. The author is very grateful to Igor for suggesting this very interesting problem, and for insightful remarks. The author was funded by a Graduate Teaching Scholarship, Department of Mathematics, King's College London. The author was supported by the Engineering \& Physical Sciences Research Council (EPSRC) Fellowship EP/M002896/1 held by Dmitry Belyaev.

\section{Kac-Rice bound: Proof of Proposition \ref{approxKRpl}}
\label{secappkr}
\subsection{Setup}
We fix a smooth $2$-dimensional sub-manifold $\Pi$ of $\mathbb{T}^3$ confined to a plane, denoting the unit normal $\overrightarrow{n}=(n_1,n_2,n_3)$. Then w.l.o.g. $\Pi$ admits everywhere the parametrisation (cf. \eqref{gammapl})
\begin{align}
\notag
\gamma:[0,A]\times[0,B]\subset\mathbb{R}^2&\to\Pi,
\\
\label{gammaplagain}
(u,v)&\mapsto p=u\xi+v\eta,
\end{align}
where $\{\overrightarrow{n},\xi,\eta\}$ is an orthonormal basis of $\mathbb{R}^3$,
\begin{equation}
\label{AB}
A:=\max\{u: u\xi+v\eta\in\Pi\},
\quad
\text{ and }
\quad
B:=\max\{v: u\xi+v\eta\in\Pi\}.
\end{equation}
Later we will choose (assuming w.l.o.g that $n_1\neq 0$)
\begin{equation}
\label{xieta}
\xi=\frac{(n_2,-n_1,0)}{\sqrt{n_1^2+n_2^2}}, \qquad\qquad \eta=\frac{(n_1n_3,n_2n_3,-n_1^2-n_2^2)}{\sqrt{n_1^2+n_2^2}}.
\end{equation}

We now introduce some necessary notation for the derivatives of the covariance function $r$ \eqref{rplane}. 
\begin{defin}
\label{thedef}
Define the row vector $D:=\nabla r$,
\begin{equation*}
D((u,v),(u',v'))=\frac{2\pi i}{N}
\sum_{\lambda\in\Lambda} e^{2\pi i\langle\lambda, (u'-u)\xi+(v'-v)\eta\rangle}\cdot\lambda
\end{equation*}
	and the Hessian matrix $H:=\text{Hess}(r)$,
	\begin{equation*}
	H((u,v),(u',v'))=-\frac{4\pi^2}{N}
	\sum_{\lambda\in\Lambda} e^{2\pi i\langle\lambda, (u'-u)\xi+(v'-v)\eta\rangle}\cdot\lambda^T\lambda.
	\end{equation*}
	We also introduce the matrix
	\begin{equation*}
	\Omega:=\begin{pmatrix}
	n_2^2+n_3^2 & -n_1n_2 & -n_1n_3
	\\ -n_1n_2 & n_1^2+n_3^2 & -n_2n_3
	\\ -n_1n_3 & -n_2n_3 & n_1^2+n_2^2
	\end{pmatrix}.
	\end{equation*}
\end{defin}

\subsection{Proof of Proposition \ref{approxKRpl}}
We bring some modifications to the proof of Proposition \cite[Proposition 1.7]{maff18}. 
With the notation of the parametrisation \eqref{gammaplagain}, consider the rectangle $U$ of vertices the origin, $A\xi$, $B\eta$, and $A\xi+B\eta$. We partition it (with boundary overlaps) into small squares $U_j$ of side length $\delta\asymp 1/\sqrt{m}$. 
\footnote{To be precise, we need $\delta\sqrt{2}<c_0/\sqrt{4\pi m/3}$, with $c_0$ as in \cite[Lemma 3.8]{maff18}.} Writing $\Pi_j:=\Pi\cap U_j$, we denote
\begin{equation*}
\mathcal{L}_{j}:=h_1(\mathcal{A}_F \cap\Pi_j)
\end{equation*}
recalling the notations $\mathcal{A}_F$ \eqref{nodset} for the nodal set and $h_1$ for Hausdorff measure. Then for \eqref{L} one has a.s.
\begin{equation*}
\mathcal{L}=\sum_j\mathcal{L}_{j}.
\end{equation*}
It follows that
\begin{equation}
\label{snspre}
\text{Var}(\mathcal{L})
=
\sum_{i,j}\text{Cov}(\mathcal{L}_i,\mathcal{L}_j).
\end{equation}

The set $\Pi^2$ is thus partitioned (with boundary overlaps) into regions $\Pi_i\times\Pi_j=:V_{i,j}$. We call the region $V_{i,j}$ {\em singular} if there are points $p\in \Pi_i$ and $p'\in \Pi_j$ s.t. $|r(p,p')|>1/2$. The union of all singular regions is the {\em singular set} $S$. It was proven in \cite[Lemma 3.12]{maff18} that
\begin{equation}
\label{Sb}
\text{meas}(S)\ll\iint_{\Pi^2}r^2(p,p')dpdp'.
\end{equation}
We separate the summation \eqref{snspre} over singular and non-singular regions:
\begin{equation}
\label{sns}
\text{Var}(\mathcal{L})
=
\sum_{V_{i,j} \text{ non-sing}} \text{Cov}(\mathcal{L}_i,\mathcal{L}_j)
+
\sum_{V_{i,j} \text{ sing}} \text{Cov}(\mathcal{L}_i,\mathcal{L}_j).
\end{equation}
In \cite[\S 3.4]{maff18} we showed the uniform bound
\begin{equation*}
\text{Cov}(\mathcal{L}_i,\mathcal{L}_j)\ll\frac{1}{m}
\end{equation*}
hence
\begin{equation}
\label{bdsingpl}
\bigg|\sum_{V_{i,j} \text{ sing}} \text{Cov}(\mathcal{L}_i,\mathcal{L}_j)\bigg|
\ll m\iint_{\Pi^2}r^2(p,p')dpdp'
\end{equation}
via \eqref{Sb}.

For non-singular regions, Kac-Rice formulae yield (see \cite[(3.19), \S 5.2, and \S 5.3]{maff18})
\begin{equation}
\label{covpl}
\text{Cov}(\mathcal{L}_i,\mathcal{L}_j)
\ll
m\iint_{V_{i,j}}
\left(
r^2
+\frac{D\Omega D^T}{m}
+\frac{tr(H\Omega H\Omega)}{m^2}
\right)
dpdp'
\end{equation}
with $D,H,\Omega$ as in Definition \ref{thedef}. We substitute \eqref{covpl} and \eqref{bdsingpl} into \eqref{sns}, and extend the domain of integration to the whole of $\Pi^2$ via another application of \eqref{Sb}. The proof of Proposition \ref{approxKRpl} is thus complete.

\section{Lattice points on spheres}
\label{seclp}
\subsection{Background}
To estimate the second moment of the covariance function $r$ and of its derivatives (the RHS of \eqref{varpl1}), we will need several considerations on lattice points on spheres $\sqrt{m}\mathcal{S}^2$. An integer $m$ is representable as a sum of three squares if and only if it is not of the form $4^l(8k+7)$, for $k,l$ non-negative integers \cite{harwri, daven1}. Recall the notation \eqref{N} $N:=|\Lambda|=r_3(m)$ for the number of such representations. Under the natural assumption $m\not\equiv 0,4,7 \pmod 8$ one has \eqref{totnumlp3}
\begin{equation*}
(\sqrt{m})^{1-\epsilon}
\ll
N
\ll
(\sqrt{m})^{1+\epsilon}.
\end{equation*}

Subtle questions about the distribution of $\Lambda/\sqrt{m}$ in the unit sphere as $m\to\infty$ are of independent interest in number theory. The limiting equidistribution of the lattice points was conjectured and proved conditionally by Linnik, and subsequently proven unconditionally \cite{duke88,dukesp,golfom}. The finer statistics of $\Lambda/\sqrt{m}$ on shrinking sets has been recently investigated by Bourgain-Rudnick-Sarnak \cite{bosaru,bsr016}.
\begin{prop}[{\cite[Theorem 1.1]{bsr016}}]
\label{asyriesz}
Fix $0<s<2$. Suppose $m\to\infty$, $m\not\equiv 0,4,7 \pmod 8$. There is some $\delta>0$ so that
\begin{equation*}
\sum_{\lambda\neq\lambda'}\frac{m^{s/2}}{|\lambda-\lambda'|^{s}}=\frac{2^{1-s}}{2-s}\cdot N^2+O(N^{2-\delta}).
\end{equation*}
\end{prop}

\subsection{Lattice points in spherical caps and segments}
\label{capseg}
In the present subsection, we collect several bounds for lattice points in certain regions of the sphere. For a more detailed account, see e.g. \cite[\S 2]{brgafa} (spherical caps) and \cite[\S\S 5,6,8]{maff3d} (spherical segments).

\begin{defin}[{\cite[Definition 4.1]{maff3d}}]
\label{defcap}
Given a sphere $\mathfrak{S}$ in $\mathbb{R}^3$ with centre $O$ and radius $R$, and a point $P\in\mathfrak{S}$, we define the \textbf{spherical cap} $\mathcal{T}$ to be the intersection of $\mathfrak{S}$ with the ball $\mathcal{B}_s(P)$ of radius $s$ centred at $P$. We will call $s$ the \textbf{radius of the cap}, and the unit vector $\alpha:=\overrightarrow{OP}/R$ the  \textbf{direction} of $\mathcal{T}$.

The intersection of $\mathfrak{S}$ with the boundary of $\mathcal{B}_s(P)$ is a circle, called the \textbf{base} of $\mathcal{T}$, and the {\bf radius of the base} will be denoted $k$. Let $Q,Q'$ be two points on the base which are diametrically opposite (note $\overline{PQ}=\overline{PQ'}=s$): we define the \textbf{opening angle} of $\mathcal{T}$ to be $\theta=\widehat{QOQ'}$. The \textbf{height} $h$ of $\mathcal{T}$ is the distance between the point $P$ and the disc base.
\end{defin}
We will be considering the sphere of radius
\begin{equation*}
R=\sqrt{m}.
\end{equation*}
If $s$, $h$, $k$ and $\theta$ denote the radius, height, radius of the base, and opening angle of $\mathcal{T}$ respectively, then geometric considerations give us the relations $0\leq s\leq 2R$, $0\leq h\leq 2R$, $0\leq k\leq R$, $0\leq \theta\leq \pi$, and
\begin{equation}
\label{shR}
s^2=2Rh.
\end{equation}

Let us introduce the notation
\begin{equation}
\label{chi(R,s)}
\chi(R,s)
=\max_\mathcal{T}\#\{\lambda\in\mathbb{Z}^3\cap \mathcal{T}\}
\end{equation}
for the maximal number of lattice points contained in any spherical cap $\mathcal{T}\subset R\mathcal{S}^2$ of radius $s$.
\begin{lemma}[Bourgain and Rudnick {\cite[Lemma 2.1]{brgafa}}]
\label{lemmachi}
We have for all $\epsilon>0$,
\begin{equation*}
\chi(R,s)
\ll
R^\epsilon\left(1+\frac{s^2}{R^{1/2}}\right)
\end{equation*}
as $R\to\infty$.
\end{lemma}
Compare Lemma \ref{lemmachi} with Conjecture \ref{brgafaconj}. We now introduce another particular region of the sphere, the segment (sometimes called `slab' or `annulus').
\begin{defin}
\label{defseg}
Given a sphere $\mathfrak{S}$ in $\mathbb{R}^3$ with centre $O$ and radius $R$, and two parallel planes $\Pi_1,\Pi_2$, we call \textbf{spherical segment} $\Gamma$ the region of the sphere delimited by $\Pi_1,\Pi_2$. The two \textbf{bases} of $\Gamma$ are the circles $\mathfrak{S}\cap\Pi_1$ and $\mathfrak{S}\cap\Pi_2$: we always assume the latter to be the larger. We define the \textbf{height} $h$ of the spherical segment to be the distance between $\Pi_1$ and $\Pi_2$. We will denote $k$ the \textbf{radius of the larger base}.

Moreover, let $\mathfrak{C}$ be a great circle of the sphere $\mathfrak{S}$, lying on a plane perpendicular to $\Pi_1$ and $\Pi_2$. Denote $\{A,B\}:=\mathfrak{S}\cap\Pi_1\cap\mathfrak{C}$ and $\{C,D\}:=\mathfrak{S}\cap\Pi_2\cap\mathfrak{C}$. We define the \textbf{opening angle} of $\mathfrak{S}$ to be $\theta=\widehat{AOC}+\widehat{BOD}=2\cdot\widehat{AOC}$. The \textbf{direction} of the spherical segment is the unit vector $\alpha$ that is the direction of the two spherical caps $\mathcal{T}_1,\mathcal{T}_2$ satisfying
$
\mathfrak{S}=\mathcal{T}_2\setminus \mathcal{T}_1.
$
\end{defin}
A cap is thus a special case of a segment. It will be convenient to always assume a spherical segment $\Gamma$ to be contained in a hemisphere, so that any two of $h,k,\theta$ completely determine $\Gamma$. We always have $0\leq h\leq R$, $0\leq k\leq R$, $0\leq \theta\leq \pi$ and the relation \cite[Lemma 5.3]{maff3d}
\begin{equation}
\label{kth}
k\theta\ll h
\end{equation}
as $R\to\infty$.

Next, we state two lemmas of \cite{maff3d} which will be needed later. 
\begin{lemma}[{\cite[Lemma 9.1]{maff3d}}]
\label{twoparpla}
Given $0<c<R$, fix a point $P\in R\mathcal{S}^2$, and let $\alpha$ be a unit vector. Then all points $P'\in R\mathcal{S}^2$ satisfying $|\langle P-P',\alpha\rangle|\leq c$ lie on the same spherical segment, of height (at most) $2c$ and direction $\alpha$ on $R\mathcal{S}^2$.
\end{lemma}

\begin{lemma}[{\cite[Lemma 7.1]{maff3d}}]
\label{cone}
Let $c=c(R)>0$, with $c\to 0$ as $R\to\infty$. Fix a point $P\in R\mathcal{S}^2$, and let $\alpha$ be a unit vector. Then all points $P'\in R\mathcal{S}^2$ satisfying $|\langle P-P',\alpha\rangle|\leq c|P-P'|$ lie: either on the same spherical segment, of opening angle $8c+O(c^3)$ and direction $\alpha$; or on the same spherical cap, of radius $\ll cR$ and direction $\alpha$, on $R\mathcal{S}^2$.
\end{lemma}

In \cite{maff3d} we found several upper bounds for the maximal number of lattice points belonging to a spherical segment $\Gamma$ of the sphere $R\mathcal{S}^2$,
\begin{equation}
\label{psi}
\psi=\psi(R,h,k,\theta):=\max_\Gamma\#\{\lambda\in\mathbb{Z}^3\cap \Gamma\},
\end{equation}
with $h,k,\theta$ as in Definition \ref{defseg}. Here we collect some of these bounds for convenience. Recall that $\kappa$ denotes the maximal number of spherical lattice points in a plane, and the types \eqref{qq}, \eqref{qr}, \eqref{rr} of vectors/planes defined in section \ref{secresults}.
\begin{prop}
Let $\Gamma\subset R\mathcal{S}^2$ be a spherical segment of opening angle $\theta$, height $h$, radius of larger base $k$, and direction $\alpha$. Then the number of lattice points lying on $\Gamma$ satisfies for every $\epsilon>0$:
\begin{enumerate}[label=(\arabic*)]
\item
if $\alpha$ is of type \eqref{qq},
\begin{equation}
\label{psi1}
\psi\ll_{\alpha} R^\epsilon\cdot(1+h);
\end{equation}
\item
if $\alpha$ is of type \eqref{qr} or \eqref{rr},
\begin{equation}
\label{psi2}
\psi\ll_{\alpha} R^{1/2+\epsilon}\cdot(R^{1/4}+h);
\end{equation}
\item
if $\alpha$ is of type \eqref{qr},
\begin{equation}
\label{psi3}
\psi\ll_{\alpha} \kappa(R)(1+R\cdot\theta^{1/2});
\end{equation}
\item
if $\alpha$ is of type \eqref{rr},
\begin{equation}
\label{psi4}
\psi\ll_{\alpha} \kappa(R)(1+R\cdot\theta^{1/3}).
\end{equation}
\end{enumerate}
\end{prop}
\begin{proof}
The bound \eqref{psi1} was proven in \cite[Proposition 6.3]{maff3d} (also see Yesha \cite[Lemma A.1]{yesh13}). We now show that \eqref{psi2} follows directly from \cite{maff3d}. Applying \cite[Proposition 5.4]{maff3d} with $\Omega=R^{1/4}$,
\begin{equation*}
\psi\ll
\chi(R,R^{1/4})
\cdot
\left\lceil\frac{k}{R^{1/4}}\right\rceil
\cdot
\left\lceil R^{3/4}\theta\right\rceil
\end{equation*}
so that, by Lemma \ref{lemmachi},
\begin{equation*}
\psi\ll
R^{\epsilon}
\cdot
\left(
1+\frac{k}{R^{1/4}}+R^{3/4}\theta+R^{1/2}k\theta
\right).
\end{equation*}
Since $0\leq k\leq R$, $0\leq \theta\leq \pi$ and $k\theta\ll h$ \eqref{kth}, we obtain \eqref{psi2}. The bounds \eqref{psi3} and \eqref{psi4} were shown in \cite[Proposition 8.3]{maff3d} and \cite[Proposition 6.2]{maff3d} respectively.
\end{proof}

\section{Proofs of Theorems \ref{thmpl} and \ref{thmc}}
\label{secpl}
\subsection{The bounds for the variance}
In this section, we prove Theorem \ref{thmpl}. We commence by further reducing our problem of bounding the variance to estimating a summation over the lattice points on the sphere. Recall the notations $\Lambda$ of the frequency set \eqref{Lambda}, $A,B\in\mathbb{R}^+$ \eqref{AB}, and vectors/matrices $D,H,\Omega$ (Definition \ref{thedef}).
\begin{lemma}
\label{2mompl}
Let $\Pi$ be a $2$-dimensional toral sub-manifold confined to a plane. Then
\begin{equation}
\label{rint}
\iint_{\Pi^2}
\left(
r^2
+\frac{D\Omega D^T}{m}
+\frac{tr(H\Omega H\Omega)}{m^2}
\right)dpdp'
\ll_\Pi\frac{1}{N}+\frac{\mathcal{G}}{N^2},
\end{equation}
where
\begin{equation}
\label{g}
\mathcal{G}=\mathcal{G}_{m,\Pi}:=\sum_{\substack{\lambda,\lambda'\in\Lambda_m\\\lambda\neq\lambda'}}
\left|\int_{0}^{A}\int_{0}^{B}e^{2\pi i\langle\lambda-\lambda',u\xi+v\eta\rangle}dudv\right|^2.
\end{equation}
\end{lemma}
The proof of Lemma \ref{2mompl} is relegated to appendix \ref{appa}. Assuming it, we deduce the following bound for the nodal intersection length variance.
\begin{cor}
Let $\Pi$ be a $2$-dimensional toral sub-manifold confined to a plane. Then
\begin{equation}
\label{varpl2}
\text{Var}(\mathcal{L})\ll_\Pi\frac{m}{N}+\frac{m}{N^2}\cdot\mathcal{G}.
\end{equation}
\end{cor}
\begin{proof}
One substitutes the estimate \eqref{rint} into the approximate Kac-Rice bound \eqref{varpl1}.
\end{proof}

In the following two lemmas we bound $\mathcal{G}$, thereby completing the proof of Theorem \ref{thmpl}. Recall that we distinguish between planes of three types, according to the unit normal $\overrightarrow{n}$ satisfying:
\begin{align}
\tag{i}
{n_2}/{n_1}\in\mathbb{Q} \quad&\text{and}\quad {n_3}/{n_1}\in\mathbb{Q};
\\
\tag{ii}
{n_2}/{n_1}\in\mathbb{Q} \quad&\text{and}\quad {n_3}/{n_1}\in\mathbb{R}\setminus\mathbb{Q};
\\
\tag{iii}
{n_2}/{n_1}\in\mathbb{R}\setminus\mathbb{Q} \quad&\text{and}\quad {n_3}/{n_1}\in\mathbb{R}\setminus\mathbb{Q}.
\end{align}
Recall further that $\kappa$ denotes the maximal number of spherical lattice points lying on a plane.
\begin{lemma}
\label{gratpl}
Let $\Pi$ be a $2$-dimensional toral sub-manifold confined to a {\em rational} plane. Then we have
\begin{equation}
\label{kbound}
\mathcal{G}
\\
\ll_\Pi N\cdot\kappa(\sqrt{m}).
\end{equation}
\end{lemma}

Lemma \ref{gratpl} will be proven in section \ref{secratpl}. For irrational planes, we have the following.
\begin{lemma}
\label{girrpl}
For every $\epsilon>0$, one has
\begin{equation}
\label{abound}
\mathcal{G}
\\
\ll_\Pi N^{1+a+\epsilon}
\end{equation}
where we may take:
\begin{enumerate}[label=(\Alph*)]
\item
\label{37}
$a=3/7$ if $\overrightarrow{n}$ is of type \eqref{qr};
\item
\label{34}
$a=3/4$ if $\overrightarrow{n}$ is of type \eqref{rr};
\item
\label{12}
$a=1/2$ conditionally on Conjecture \ref{brgafaconj}.
\end{enumerate}
\end{lemma}

Lemma \ref{girrpl} will be proven in sections \ref{secirrpl} and \ref{secc}. Assuming them we may complete the proofs of our main theorems.
\begin{proof}[Proof of Theorems \ref{thmpl} and \ref{thmc} assuming Lemmas \ref{gratpl} and \ref{girrpl}]
One substitutes \eqref{kbound} into \eqref{varpl2} to obtain \eqref{varplr}. One substitutes \eqref{abound} into \eqref{varpl2} to obtain \eqref{varpli} and \eqref{varplc}.
\end{proof}

\subsection{Rational planes}
\label{secratpl}
In this subsection we prove Lemma \ref{gratpl}. We will need a preparatory result, the proof of which will follow in appendix \ref{appa}.
\begin{lemma}
\label{trineq}
Let $\xi,\eta\in\mathbb{R}^3$, satisfying
\begin{equation*}
\langle\lambda-\lambda',\xi\rangle\cdot\langle\lambda-\lambda',\eta\rangle\neq 0.
\end{equation*}
Then
\begin{equation}
\label{min}
\left|\int_{0}^{A}\int_{0}^{B}e^{2\pi i\langle\lambda-\lambda',u\xi+v\eta\rangle}dudv\right|^2
\ll\min\left(1,\frac{1}{\langle\lambda-\lambda',\xi\rangle^2\langle\lambda-\lambda',\eta\rangle^2}\right).
\end{equation}
\end{lemma}

\begin{proof}[Proof of Lemma \ref{gratpl} assuming Lemma \ref{trineq}]
We split the summation
\begin{equation*}
\mathcal{G}
=\sum_{\lambda\neq\lambda'}\left|\int_{0}^{A}\int_{0}^{B}e^{2\pi iu\langle\lambda-\lambda',\xi\rangle}du\cdot e^{2\pi iv\langle\lambda-\lambda',\eta\rangle}dv\right|^2
\end{equation*}
over the set of pairs $(\lambda,\lambda')$ s.t. $\langle\lambda-\lambda',\xi\rangle\cdot\langle\lambda-\lambda',\eta\rangle\neq 0$ and its complement. Thanks to the bounds \eqref{min} of Lemma \ref{trineq},
\begin{align}
\label{ratsum}
\notag
\mathcal{G}\ll_{\Pi}
&\#\{(\lambda,\lambda'): |\langle\lambda-\lambda',\xi\rangle|=0 \ \vee \  |\langle\lambda-\lambda',\eta\rangle|=0\}
\\&+\sum_{\langle\lambda-\lambda',\xi\rangle\cdot\langle\lambda-\lambda',\eta\rangle\neq 0}\frac{1}{\langle\lambda-\lambda',\xi\rangle^2\langle\lambda-\lambda',\eta\rangle^2}.
\end{align}

We claim that there are few pairs $(\lambda,\lambda')$ satisfying $\langle\lambda-\lambda',\xi\rangle=0$. Indeed, once we fix $\lambda$, the lattice point $\lambda'$ is confined to the plane
\begin{equation}
\label{plane}
\langle\xi,(x,y,z)\rangle=l,
\end{equation}
where $l:=\langle\lambda,\xi\rangle\in\mathbb{R}$. By definition of $\kappa$, there are at most $\kappa(\sqrt{m})$ solutions $(x,y,z)\in\Lambda$ to \eqref{plane}. Therefore,
\begin{equation}
\label{ratsum1}
\#\{(\lambda,\lambda'): |\langle\lambda-\lambda',\xi\rangle|=0\}
=\sum_{\lambda\in\Lambda}
\#\{\lambda': \ \langle\lambda',\xi\rangle=\langle\lambda,\xi\rangle\}
\leq N\cdot\kappa(\sqrt{m}).
\end{equation}
Similarly, there are few pairs $(\lambda,\lambda')$ such that $\langle\lambda-\lambda',\eta\rangle=0$. 

We turn to bounding the summation in \eqref{ratsum}. By assumption, $\overrightarrow{n}$ is of type \eqref{qq}. Taking $\xi,\eta$ as in \eqref{xieta}, then $\xi,\eta$ are also of type \eqref{qq}, hence we may write $\xi=c\tilde{\xi}$ and $\eta=c'\tilde{\eta}$, where $\tilde{\xi},\tilde{\eta}\in\mathbb{Z}^3$ and $c,c'$ are real numbers. Therefore,
\begin{multline*}
\sum_{\langle\lambda-\lambda',\xi\rangle\cdot\langle\lambda-\lambda',\eta\rangle\neq 0}\frac{1}{\langle\lambda-\lambda',\xi\rangle^2\langle\lambda-\lambda',\eta\rangle^2}
\\\ll_\Pi\sum_{\lambda}\sum_{a\neq 0}\sum_{b\neq 0}\frac{1}{a^2}\frac{1}{b^2}
\cdot\#\{\lambda':\langle \tilde{\xi},\lambda'\rangle=a\in\mathbb{Z} \ \wedge \ \langle \tilde{\eta},\lambda'\rangle=b\in\mathbb{Z}\}.
\end{multline*}

For fixed $a,b$, the lattice point $\lambda'$ is confined to the intersection of the two planes
\begin{equation*}
\langle\tilde{\xi},\lambda'\rangle=a \qquad \text{ and } \qquad \langle\tilde{\eta},\lambda'\rangle=b.
\end{equation*}
Since $\tilde{\xi}\perp\tilde{\eta}$, these two planes intersect in a line, hence the number of solutions $\lambda'\in\Lambda$ cannot exceed two. It follows that
\begin{gather}
\label{ratsum2}
\sum_{\langle\lambda-\lambda',\xi\rangle\cdot\langle\lambda-\lambda',\eta\rangle\neq 0}\frac{1}{\langle\lambda-\lambda',\xi\rangle^2\langle\lambda-\lambda',\eta\rangle^2}
\ll N.
\end{gather}
Substituting \eqref{ratsum1} and \eqref{ratsum2} into \eqref{ratsum} yields \eqref{kbound}.
\end{proof}

\subsection{Irrational planes}
\label{secirrpl}
In the present subsection we prove Lemma \ref{girrpl} parts \ref{37} and \ref{34}, using the bounds for lattice points in spherical caps and segments of section \ref{capseg}. We introduce the parameters $c=c(N),\rho=\rho(N)>0$ and consider the three regimes
\begin{itemize}
\item
first regime: $|\langle\lambda-\lambda',\xi\rangle|\leq c$;
\item
second regime: 
$|\langle\lambda-\lambda',\eta\rangle|\leq \rho|\lambda-\lambda'|$;
\item
third regime: $|\langle\lambda-\lambda',\xi\rangle|\geq c$, $|\langle\lambda-\lambda',\eta\rangle|\geq \rho|\lambda-\lambda'|$.
\end{itemize}
We apply the bounds \eqref{min} of Lemma \ref{trineq} to obtain
\begin{multline}
\label{irrsum}
\mathcal{G}\ll_{\Pi}
\#\{(\lambda,\lambda'): |\langle\lambda-\lambda',\xi\rangle|\leq c\}
+\#\{(\lambda,\lambda'): |\langle\lambda-\lambda',\eta\rangle|\leq \rho|\lambda-\lambda'|\}
\\+\sum_{\substack{
|\langle\lambda-\lambda',\xi\rangle|\geq c
\\
|\langle\lambda-\lambda',\eta\rangle|\geq \rho|\lambda-\lambda'|
}}\frac{1}{\langle\lambda-\lambda',\xi\rangle^2\langle\lambda-\lambda',\eta\rangle^2}.
\end{multline}

\begin{enumerate}[label=(\Alph*), listparindent=\the\parindent]
\item
\label{parta}
Let $\overrightarrow{n}$ be of type \eqref{qr}. Taking $\xi,\eta$ as in \eqref{xieta}, then $\xi$ is of type \eqref{qq} and $\eta$ of type \eqref{qr}.

\underline{First regime}. Once we fix $\lambda$, the lattice points $\lambda'$ satisfying
\begin{equation*}
|\langle\lambda-\lambda',\xi\rangle|\leq c
\end{equation*}
lie on a spherical segment $\Gamma_\lambda$ of height at most $2c$ and direction $\xi$ (see Lemma \ref{twoparpla}). As $\xi$ is of type \eqref{qq}, we may apply \eqref{psi1}:
\begin{equation}
\label{1reg}
\#\{(\lambda,\lambda'): |\langle\lambda-\lambda',\xi\rangle|\leq c\}
\ll N R^\epsilon\left(1+c\right).
\end{equation}

\underline{Second regime}. Once we fix $\lambda$, the lattice points $\lambda'$ satisfying
\begin{equation*}
|\langle\lambda-\lambda',\eta\rangle|\leq \rho|\lambda-\lambda'|
\end{equation*}
lie on a spherical segment $\Gamma_{\lambda}$ of opening angle $8\rho+O(\rho^3)$ and direction $\eta$, or on a spherical cap $\mathcal{T}_\lambda$ of radius $\ll\rho R$ and direction $\eta$, on $R\mathcal{S}^2$ (see Lemma \ref{cone}). Later we are going to choose $\rho=N^{-8/7}$, thus the number of lattice points in $\mathcal{T}_\lambda$ of radius $\rho R=o(1)$ is $\ll R^{\epsilon}$. To control the lattice points in each $\Gamma_{\lambda}$, as $\eta$ is of type \eqref{qr}, we may apply \eqref{psi3}:
\begin{equation}
\label{2reg}
\#\{(\lambda,\lambda'): |\langle\lambda-\lambda',\eta\rangle|\leq \rho\cdot|\lambda-\lambda'|\}
\ll N R^\epsilon(1+R\rho^{1/2}).
\end{equation}

\underline{Third regime}. Here we have
\begin{equation}
\label{3reg}
\sum\frac{1}{\langle\lambda-\lambda',\xi\rangle^2\langle\lambda-\lambda',\eta\rangle^2}
\leq\frac{1}{c^2\rho^2}\sum\frac{1}{|\lambda-\lambda'|^{2-\epsilon'}}\ll\frac{m^\epsilon}{c^2\rho^2}
\end{equation}
via an application of Proposition \ref{asyriesz}. Collecting the estimates \eqref{1reg}, \eqref{2reg}, \eqref{3reg}, and \eqref{irrsum} we obtain
\begin{equation*}
\mathcal{G}
\ll_\Pi N R^\epsilon\left(1+c\right)+N R^\epsilon(1+R\rho^{1/2})+\frac{m^\epsilon}{c^2\rho^2}.
\end{equation*}
The optimal choice of parameters $(c,\rho)=(N^{3/7},N^{-8/7})$ yields \eqref{abound} with $a=3/7$.

\item
In case $\overrightarrow{n}$ is of type \eqref{rr}, then $\xi$ is of type \eqref{qr} and $\eta$ of type \eqref{rr}. After a relabelling \footnote{Alternatively, one could swap the roles of $\xi,\eta$ when defining the three regimes.}, $\xi$ is of type \eqref{rr} and $\eta$ of type \eqref{qr}. We modify the proof of part \ref{parta} in the following way. In the first regime, by Lemma \ref{twoparpla} and \eqref{psi2},
\begin{equation*}
\#\{(\lambda,\lambda'): |\langle\lambda-\lambda',\xi\rangle|\leq c\}
\ll N R^{1/2+\epsilon}\cdot(R^{1/4}+c).
\end{equation*}

In the second regime, the lattice points in the cap $\mathcal{T}_\lambda$ of radius $\ll\rho R$ have the upper bound $R^\epsilon(1+\rho^2 R^{3/2})$ (Lemma \ref{lemmachi}), while those in each segment $\Gamma_\lambda$ are no more than $R^\epsilon(1+R\rho^{1/2})$ \eqref{psi3}. It follows that
\begin{multline*}
\#\{(\lambda,\lambda'): |\langle\lambda-\lambda',\eta\rangle|\leq \rho|\lambda-\lambda'|\}
\\\ll N R^\epsilon(1+\rho^2 R^{3/2})
+N R^\epsilon(1+R\rho^{1/2}).
\end{multline*}

Choosing e.g. $(c,\rho)=(N^{1/14},N^{-6/7})$, we have obtained the bound
\begin{equation*}
\mathcal{G}
\ll_\Pi N R^\epsilon(1+\rho^2 R^{3/2}+R\rho^{1/2})+N R^{1/2+\epsilon}(R^{1/4}+c)+\frac{1}{c^2\rho^2}\ll N^{7/4+\epsilon}
\end{equation*}
proving Lemma \ref{girrpl} part \ref{34}.
\end{enumerate}

\subsection{Conditional result}
\label{secc}
It remains to show Lemma \ref{girrpl} part \ref{12}. Assuming Conjecture \ref{brgafaconj}, one may improve the bound \eqref{psi2} for lattice points in spherical segments of given height and larger base radius.
\begin{cor}[{\cite[Corollary 5.6]{maff3d}}]
\label{covercapscor2}
Assume Conjecture \ref{brgafaconj}. Let $\Gamma\subset R\mathcal{S}^2$ be a spherical segment of height $h$ and radius of larger base $k$. Then for every $\epsilon>0$,
\begin{equation}
\label{psi5}
\psi\ll R^{\epsilon}
\cdot
(R^{1/2}+h).
\end{equation}
\end{cor}
We introduce the parameters $c=c(N),c'=c'(N)>0$ and consider the three regimes
\begin{itemize}
\item
first regime: $|\langle\lambda-\lambda',\xi\rangle|\leq c$;
\item
second regime:
$|\langle\lambda-\lambda',\eta\rangle|\leq c'$;
\item
third regime: $|\langle\lambda-\lambda',\xi\rangle|\geq c$, $|\langle\lambda-\lambda',\eta\rangle|\geq c'$.
\end{itemize}
We apply the bounds \eqref{min} of Lemma \ref{trineq} to obtain
\begin{multline}
\label{csum}
\mathcal{G}\ll_{\Pi}
\#\{(\lambda,\lambda'): |\langle\lambda-\lambda',\xi\rangle|\leq c\}
+\#\{(\lambda,\lambda'): |\langle\lambda-\lambda',\eta\rangle|\leq c'\}
\\+\sum_{\substack{
|\langle\lambda-\lambda',\xi\rangle|\geq c
\\
|\langle\lambda-\lambda',\eta\rangle|\geq c'
}}\frac{1}{\langle\lambda-\lambda',\xi\rangle^2\langle\lambda-\lambda',\eta\rangle^2}.
\end{multline}

\underline{First regime}. Once we fix $\lambda$, the lattice points $\lambda'$ satisfying
\begin{equation*}
|\langle\lambda-\lambda',\xi\rangle|\leq c
\end{equation*}
lie on a spherical segment $\Gamma_\lambda$ of height at most $2c$ and direction $\xi$ (see Lemma \ref{twoparpla}). By \eqref{psi5},
\begin{equation}
\label{1regc}
\#\{(\lambda,\lambda'): |\langle\lambda-\lambda',\xi\rangle|\leq c\}
\ll N R^\epsilon(R^{1/2}+c).
\end{equation}

\underline{Second regime}. Similarly to the first regime,
\begin{equation}
\label{2regc}
\#\{(\lambda,\lambda'): |\langle\lambda-\lambda',\eta\rangle|\leq c'\}
\ll N R^\epsilon(R^{1/2}+c').
\end{equation}

\underline{Third regime}. Here we simply write
\begin{equation}
\label{3regc}
\sum\frac{1}{\langle\lambda-\lambda',\xi\rangle^2\langle\lambda-\lambda',\eta\rangle^2}
\leq\frac{N^2}{c^2c'^2}.
\end{equation}
Collecting the estimates \eqref{1regc}, \eqref{2regc}, \eqref{3regc}, and \eqref{csum}, we obtain
\begin{equation*}
\mathcal{G}
\ll_\Pi N R^\epsilon(R^{1/2}+c+c')+\frac{N^2}{c^2c'^2}\ll N^{3/2+\epsilon},
\end{equation*}
choosing e.g. $c=c'=N^{1/5}$. This completes the proof of Lemma \ref{girrpl} part \ref{12}.



\appendix

\section{Proofs of auxiliary results}
\label{appa}
\noindent
In this appendix, we prove a couple of auxiliary lemmas.
\begin{proof}[Proof of Lemma \ref{2mompl}]
We follow \cite[\S 3 and \S 6]{maff2d} and \cite[\S 3]{maff3d}. Squaring $r$ we obtain
\begin{equation*}
r^2((u,v),(u',v'))
=
\frac{1}{N^2}\sum_{\lambda,\lambda'} e^{2\pi i\langle\lambda-\lambda',(u'-u)\xi+(v'-v)\eta\rangle}
\end{equation*}
and on integrating over $\Pi^2$, the contribution of the diagonal terms to \eqref{rint} is
\begin{equation}
\label{diag}
\frac{1}{N^2}\int_{0}^{A}\int_{0}^{B}\int_{0}^{A}\int_{0}^{B}\sum_{\lambda}1dudvdu'dv'
\ll\frac{1}{N}.
\end{equation}

The off-diagonal terms equal
\begin{align}
\label{od}
\notag
&\int_{([0,A]\times[0,B])^2}\frac{1}{N^2}\sum_{\lambda\neq\lambda'}e^{2\pi i\langle\lambda-\lambda',(u'-u)\xi+(v'-v)\eta\rangle}
dudvdu'dv'
\\&\notag=\frac{1}{N^2}\sum_{\lambda\neq\lambda'}
\int_{0}^{A}\int_{0}^{B}e^{2\pi i\langle\lambda-\lambda',u'\xi+v'\eta\rangle}dudv
\int_{0}^{A}\int_{0}^{B}e^{-2\pi i\langle\lambda-\lambda',u\xi+v\eta\rangle}du'dv'
\\&=\frac{1}{N^2}\sum_{\lambda\neq\lambda'}
\left|\int_{0}^{A}\int_{0}^{B}e^{2\pi i\langle\lambda-\lambda',u\xi+v\eta\rangle}dudv\right|^2=\frac{\mathcal{G}}{N^2}.
\end{align}
By \eqref{diag} and \eqref{od},
\begin{equation*}
\iint_{\Pi^2}
r^2dpdp'
\ll_\Pi\frac{1}{N}+\frac{\mathcal{G}}{N^2}.
\end{equation*}

To complete the proof of \eqref{rint}, by the symmetries it will suffice to show that
\begin{equation}
\label{suffice}
\iint_{\Pi^2}
\left(
\frac{r_u^2}{m}
+\frac{r_{uu'}^2}{m^2}
\right)dpdp'
\ll_\Pi\frac{1}{N}+\frac{\mathcal{G}}{N^2}
\end{equation}
(see Definition \ref{thedef}). One has
\begin{equation*}
r_u=\frac{2\pi i}{N}\sum_{\lambda\in\Lambda}\langle\lambda,\xi\rangle e^{2\pi i\langle\lambda,(u'-u)\xi+(v'-v)\eta\rangle},
\end{equation*}
hence, as required in \eqref{suffice},
\begin{align*}
&\iint_{\Pi^2}\frac{r_u^2}{m}dpdp'
\ll_\Pi\frac{1}{N}+\int_{([0,A]\times[0,B])^2}\frac{1}{N^2}\sum_{\lambda\neq\lambda'}\left\langle\frac{\lambda}{|\lambda|},\xi\right\rangle\left\langle\frac{\lambda'}{|\lambda'|},\xi\right\rangle
\\
&\cdot e^{2\pi i\langle\lambda-\lambda',(u'-u)\xi+(v'-v)\eta\rangle}
dudvdu'dv'
\\&\leq\frac{1}{N}+\int_{([0,A]\times[0,B])^2}\frac{1}{N^2}\sum_{\lambda\neq\lambda'}e^{2\pi i\langle\lambda-\lambda',(u'-u)\xi+(v'-v)\eta\rangle}
dudvdu'dv'
\\&=\frac{1}{N}+\frac{\mathcal{G}}{N^2}
\end{align*}
where in the first inequality we isolated the diagonal terms and in the second we applied Cauchy-Schwartz. The calculation for the second derivatives is very similar and we omit it here.
\end{proof}
\begin{proof}[Proof of Lemma \ref{trineq}]
The first upper bound in \eqref{min} is a straightforward application of the triangle inequality. To show the second bound in \eqref{min}, we integrate and apply the triangle inequality,
\begin{equation*}
\left|\int_{0}^{A} e^{2\pi iu\langle\lambda-\lambda',\xi\rangle}du\right|^2
=\frac{|e^{2\pi i A\langle\lambda-\lambda',\xi\rangle}-1|}{4\pi^2\langle\lambda-\lambda',\xi\rangle^2}
\leq\frac{1}{\pi^2}\cdot\frac{1}{\langle\lambda-\lambda',\xi\rangle^2}
\end{equation*}
and similarly for the integral over $[0,B]$. This completes the proof of Lemma \ref{trineq}.
\end{proof}

\addcontentsline{toc}{section}{References}
\bibliographystyle{plain}
\bibliography{bibfile}

\Addresses

\end{document}